\documentclass{amsart}
\usepackage{amssymb,latexsym}
\usepackage{color}
\theoremstyle{plain}
\newtheorem{theorem}{Theorem}

\newtheorem{lemma}{Lemma}

\theoremstyle{definition}

\newtheorem{remark}{Remark}

\DeclareMathOperator{\reg}{reg}
\DeclareMathOperator{\Res}{Res}

\newcommand{\enm}[1]{\ensuremath{#1}} 

\newcommand{\PP}{\enm{\mathbb {P}}}
\newcommand{\NN}{\enm{\mathbb {N}}}

\newcommand{\Oo}{\enm{\mathcal {O}}}
\newcommand{\Ii}{\enm{\mathcal {I}}}

\newcommand{\Bb}{\enm{\mathcal {B}}}
\newcommand{\Ll}{\enm{\mathcal{L}}}
\newcommand{\Ww}{\enm{\mathcal {W}}}

\newcommand{\Dd}{\enm{\mathcal {D}}}
\newcommand{\Jj}{\enm{\mathcal {J}}}

\begin{document}

\title[Rational curves of degree $12$ on a general quintic threefold]
{Finiteness of rational curves of degree $12$ \\ on a general quintic threefold}
\author{Edoardo Ballico and Claudio Fontanari}
\address{Dipartimento di Matematica\\
 Universit\`a di Trento\\
38123 Povo (TN), Italy}
\email{edoardo.ballico@unitn.it, claudio.fontanari@unitn.it}
\thanks{}

\begin{abstract}
We prove the following statement, predicted by Clemens' conjecture:
\emph{A generic quintic threefold contains only finitely many smooth rational curves of degree $12$.}

\end{abstract}

\maketitle

\section{Introduction}

The present paper is entirely devoted to the proof of the following instance of Clemens' conjecture (\cite{c}):

\begin{theorem}\label{clemens}
A generic quintic threefold contains only finitely many smooth rational curves of degree $12$. 
\end{theorem}

We point out that the cases $d \le 11$ have been previously addressed in \cite{k} ($d \le 7$), \cite{n} and \cite{jk} 
($d=8,9$), \cite{c1} ($d=10$), \cite{c2} and \cite{da} ($d=11$), and we recall the general set-up.  

Let $M_d$ be the set of smooth rational curves of degree $d$ in $\PP^4$. It is smooth and irreducible of dimension $5d+1$. Let $\PP^{125}$ denote the projective space of all quintic hypersurfaces of $\PP^4$ and consider the incidence correspondence $I_d = \{ (C,W): C \subset W \} \subset M_d\times \PP^{125}$. Let $\pi _1: I_d\to M_d$ and $\pi _2: I_d\to \PP^{125}$ denote the restrictions to $I_d$ of the two projections. 

The map $\pi _2$ turns out to be finite if for every irreducible family $\Gamma \subseteq M_d$ with general element $C$ 
we have: 
\begin{equation}\label{one}
\dim \Gamma + (h^0(\Ii _C(5))-1) \le 125.
\end{equation}
From the standard exact sequence 
$$
0 \to H^0(\Ii _C(5)) \to H^0(\Oo _{\PP^4}(5)) \to H^0(\Oo _C(5)) \to  H^1(\Ii _C(5)) \to 0
$$
it follows that 
\begin{equation}\label{two}
h^0(\Ii _C(5))-1 = 125 - (5d+1) + h^1(\Ii _C(5))
\end{equation}
and by \cite{be} the general curve $C$ in $M_d$ satisfies $h^1(\Ii _C(5)) =0$, so in order to prove 
Clemens' conjecture one needs to control curves $C$ with $h^1(\Ii _C(5)) > 0$.

In the case $d=12$, we show that if an irreducible family $\Gamma \subseteq M_{12}$ of non-degenerate 
curves is a potential exception to Clemens' conjecture, then its general element $C$ satisfies 
$h^1(\Ii _C(2)) \ge 13$. It follows that $h^0(\Ii _C(2)) \ge 3$ and this provides a contradiction 
(see Lemma \ref{0a1}). 

The key point in our 
reduction is to obtain $h^1(\Ii _C(2)) \ge 13$ from $h^1(\Ii _C(5)) > 0$. Indeed, Lemma \ref{nn1} implies 
that $h^1(\Ii _C(t-1)) \ge 4 + h^1(\Ii _C(t))$ except in two special cases, which are identified 
by Lemma \ref{c1.1} and then excluded in Lemmas \ref{0a8}, \ref{0a9}, \ref{0a11}, \ref{0a12}. 
Finally, a careful analysis of the degenerate case is provided (see Section \ref{deg}).

We remark that the strong form of Clemens' conjecture (as proved by Cotterill in \cite{c1} and \cite{c2} 
for $d=10, 11$, characterizing also singular irreducible rational curves on the general quintic threefold) 
cannot be achieved by our methods.    

We work over the complex field $\mathbb{C}$.

We thank Gilberto Bini, Gianluca Occhetta  and Luis Sol\'a Conde for helpful conversations.
We are also grateful to the anonymous referee and to Carlo Madonna for useful comments.

This research was partially supported by FIRB 2012 "Moduli spaces and Applications", by GNSAGA of INdAM and 
by MIUR (Italy).

\section{Non-degenerate case}\label{non}

\begin{lemma}\label{0a1}
If $C\in M_{12}$, $C$ is non-degenerate and $h^0(\Ii _C(2)) \ge 3$, then there is no smooth quintic $3$-fold containing $C$.
\end{lemma}

\begin{proof}
Assume by contradiction $h^0(\Ii _C(2)) \ge 3$ and the existence of a smooth quintic $3$-fold $W\subset \PP^4$ with $W\supset C$ and let $E\subset \PP^4$ be the intersection of $3$ general element of 
$|\Ii _C(2)|$. Since $\deg ({C})=12 >8$, Bezout theorem gives
the existence of an integral surface $F$ such that $C\subset F\subseteq E$. 
Since $C$ is non-degenerate, $F$ is non-degenerate and
so $\deg (F)\ge 3$. Assume $E = F$, i.e. $\deg (F)=4$. Since the complete intersection of $2$ quadric hypersurfaces is contained in exactly two linearly independent quadrics
and $\deg ({C}) > 8$, we get $h^0(\Ii _C(2)) =2$, a contradiction.
Thus $\deg (F) =3$. The classification of minimal degree non-degenerate surfaces in $\PP^4$ gives $h^0(\Ii _F(2)) =3$. By assumption there is $W$ with $C\subset W$.
Since $\mathrm{Pic}(W)$ is freely generated by $\Oo _W(1)$, $F\nsubseteq W$. Hence $W\cap F$ links $C$ to a degree $3$ locally Cohen-Macaulay curve $T\subset W\cap F$.
By the classification of minimal degree surfaces in $\PP^4$, either $F$ is a cone with vertex $o$ over a rational normal curve $D\subset \PP^3$ or
$F$ is isomorphic to the the Hirzebruch surface $F_1$ embedded by the complete linear system $|h+2f|$, where $h$ is a section of the ruling of $F_1$ and $f$ is a fiber of the ruling of $F_1$. 

First assume that $F$ is a cone. Since $C$ is smooth, it has multiplicity at most $1$ at $o$. Hence $o\notin C$ and the linear projection from $o$ induces a degree $4$ map $\ell: C\to D$.
Let $\pi : G\to F$ be the blowing up of $o$ and $C'$ the strict transform of $C$ in $G$. $G$ is isomorphic to the Hirzebruch surface $F_3$ and the map $\pi$ is induced by the complete linear  system $|h+3f|$. Since $o\notin C$ and $\deg (\ell )=4$, $\pi$ induces an isomorphism $C'\to C$  and $C'\in |4h+12f|$. We have
$\omega _G\cong \Oo _G(-2h-5f)$. The adjunction formula gives $\omega _{C'} \cong \Oo _G(2h+7f)$. Hence $h^0(\omega _{C'})>0$, contradicting the rationality and smoothness
of $C'$.

Now assume $F\cong F_1$. Take $a, b\in \NN$ such that $C\in |ah+bf|$. Since $C$ is irreducible and not a line, we have $b\ge a>0$. Since $\Oo _C(1) \cong \Oo _C(h+2f)$, $h^2=-1$, $h\cdot f=1$, $f^2=0$ and $\deg ({C}) =d$,
we have $12 = a+b$. Since $\omega _{F_1}\cong \Oo _{F_1}(-2h-3f)$, the adjunction formula gives $\omega _C \cong \Oo _C((a-2)h+(b-3)f)$. Since
$\deg (\omega _C)=-2$, we get $(ah+bf)\cdot ((a-2)h+(b-3)f) =-2$, i.e. $-a(a-2) +a(b-3) +b(a-2) =-2$, i.e $(b-a)(a-2) +a(b-3) =-2$. Since $b\ge a>0$ and $b=12-a$, we get $a=1$
and $b=11$. Since $\Oo _{F_1}(5) \cong \Oo _{F_1}(5h+10f)$ and $b=11$, we have $h^0(F_1,\Ii _{C,F_1}(5)) =0$. Hence $W\supset F_1$, a contradiction.
\end{proof}

The following fact is one of key ingredients in the proof of \cite[Theorem 4.1]{da}.

\begin{lemma}\label{nn1}
Fix integer $t\ge 2$, $r\ge 3$ and an integral and non-degenerate curve $T\subset \PP^r$ such that $h^1(\Ii _T(t)) >0$. Assume that $h^1(M,\Ii _{M\cap T,M}(t)) =0$
for every hyperplane $M\subset \PP^r$. Then $h^1(\Ii _T(t-1)) \ge r+h^1(\Ii _T(t))$.
\end{lemma}

\begin{proof}
For any hyperplane $M\subset \PP^r$ we have an exact sequence
\begin{equation}\label{eqnn1}
0 \to \Ii _T(t-1) \to \Ii _T(t) \to \Ii _{T\cap M,M}(t)\to 0
\end{equation}
Since $h^1(M,\Ii _{T,M}(t)) =0$, the map $H^1(\Ii _T(t-1)) \to H^1(\Ii _T(t))$ is surjective, hence its dual
$e_M: H^1(\Ii _T(t))^\vee \to H^1(\Ii _T(t-1))^\vee$ is injective. Taking
the equations of all hyperplanes we get a bilinear map map $u: H^1(\Ii _T(t))^\vee \times H^0(\Oo _{\PP^4}(1)) \to H^1(\Ii _T(t-1))^\vee$, which is injective
with respect to the second variables, i.e. for every non-zero linear form $\ell $ $u_{|H^1(\Ii _T(t))^\vee \times \{\ell \}}$ is injective (it is $e_M$ with $M:= \{\ell =0\}$).
Hence if $(a,\ell )\in H^1(\Ii _T(t))^\vee \times H^0(\Oo _{\PP^4}(1))$ with $a\ne 0$ and $\ell \ne 0$, then $u(a, \ell ) = e_M(a) \ne 0$. Therefore the bilinear map $u$ is non-degenerate in each variable.
Hence $h^1(\Ii _T(t-1)) \ge h^1(\Ii _T(t)) +h^0(\Oo _{\PP^r}(1)) -1$ by the bilinear lemma.
\end{proof}

The next Lemma \ref{c1.1} is perhaps the technical heart of this work. 
It relies on a particular case of a very strong result on $0$-dimensional schemes in the plane, namely, 
\cite[Corollaire 2]{ep} (see also \cite[Remarque (i)]{ep}). We recall the statement in \cite{ep} for reader's convenience.  
Let $E \subset \mathbb{P}^2$ be a zero-dimensional scheme of degree $d$. Let $\tau := \max \{n: h^1(\mathcal{I}_E(n) > 0 \}$. Let $s$ be an integer such that $s \le d/s$ and $\tau \ge s-3+d/s$. 
Then either $E$ is the complete intersection of a curve of degree $s$ and a curve of degree $d/s$ and $\tau=s-3+d/s$, or there exists $s'$ with $0 < s' < s$ and a subscheme $E' \subset E$ contained 
in a curve of degree $s'$ such that $s'(\tau +(5-s')/2) \ge \deg(E') \ge s'(\tau-s'+3)$. In particular, if $\tau > d/3$, then either we have $\tau+2$ points on a line (counted with multiplicity), or we have 
$2\tau+2$ or $2\tau+3$ points on a conic (counted with multiplicity). 

For the proof of Lemma \ref{c1.1} we also need to introduce the notion of residual scheme. Let $M$ be a projective scheme, $A$ a closed subscheme and $D\subset M$ an effective Cartier divisor of $M$. 
The residual scheme $\mathrm{Res}_D(A)$ of $A$ with respect to $D$ is the closed subscheme of $M$ with $\Ii _A:\Ii _D$ as its ideal sheaf. We always have $\mathrm{Res}_D(A) \subseteq A$. If $A$ 
is a reduced scheme, then $\mathrm{Res}_D(A)$ is the union of the irreducible components of $A$ not contained in $D$. If $A$ is a zero-dimensional scheme, then $\deg (A) =\deg (A\cap D)+\deg (\mathrm{Res}_D(A))$. 
For any line bundle $\Ll$ on $M$ we have an exact sequence
$$
0 \to \Ii _{\mathrm{Res}_D(A)}\otimes \Ll (-D)\to \Ii _A\otimes \Ll \to \Ii _{A\cap D,D}\otimes (\Ll _{|D})\to 0.
$$

\begin{lemma}\label{c1.1}
Fix an integer $t \ge 2$. Set $M:= \PP^3$ and let $Z\subset M$ a  zero-dimensional scheme spanning $M$ and with $\deg (Z) \le 3t$. We have
$h^1(M,\Ii _{Z,M}(t)) \ne 0$ if and only if either there is a line $R\subset M$ with $\deg (R\cap Z)\ge t+2$ or there is a conic $D\subset M$
such that $\deg (D\cap Z)\ge 2t+2$ or there is a line $L\subset M$ such that $\deg (Z\cap L) =t+1$ and the union of the connected components of $Z$ whose reduction
is contained in $L$ has degree $\ge 2t+2$.
\end{lemma}

\begin{proof}
Set $Z_0:= Z$. Let $N_1\subset M$ be a plane such that $e_1:= \deg (Z\cap N_1)$ is maximal. Set $Z_1:= \Res _{N_1}(Z_0)$. For each integer
$i\ge 2$ define recursively the plane $N_i$, the integer $e_i$ and the zero-dimensional scheme $Z_i$ in the following way. Let $N_i\subset M$ be any hyperplane
such that $e_i:= \deg (Z_{i-1}\cap N_i)$ is maximal. Set $Z_i:= \Res _{N_i}(Z_{i-1})$. For each $i\ge 1$ we have an exact sequence
\begin{equation}\label{eq1}
0 \to \Ii _{Z_i}(t-i) \to \Ii _{Z_{i-1}}(t+1-i) \to \Ii _{Z_{i-1}\cap N_i,N_i}(t+1-i) \to 0
\end{equation}
We have $e_i\ge e_{i-1}$ for all $i$. Since any degree $3$ subscheme of $M$
is contained in a plane, if $e_i\le 2$, then $Z_{i-1}\subset N_i$ and $Z_i=\emptyset$. Since $\deg (Z) \le 3t$, there is an integer $i$ such that $1\le i \le t$ and $Z_i=\emptyset$.
From (\ref{eq1}) we get an integer $i\in \{1,\dots ,t\}$ such that $h^1(N_i,\Ii _{Z_{i-1}\cap N_i,N_i}(t+1-i)) >0$. 
Indeed, the fact that $Z_i$ is empty for some index $1 \le i \le t$ forces the cohomology of the ideal sheaf of $Z_i$ to be that of the ambient projective plane $N_i$.
We call $c$ the minimal integer $i$. Since $h^1(N_c,\Ii _{Z_{i-1}\cap N_c,N_c}(t+1-c)) >0$, either $\deg (Z_{i-1}\cap N_c,N_c)) \ge 2(t+1-c)+2$ or there is a line $L$ with
 $\deg (L_c\cap Z_{i-1}\cap N_c)\ge t+3-c$ (\cite[Lemma 34]{bgi}). In particular, since $c\le t$, we have $e_c\ge t+3-c$.
Since the sequence $\{e_i\}_{i\ge 1}$ is non-increasing, we have $ce_c \ge c(t+3-c)$. Since $\sum _{i\ge 1} e_i =\deg (Z)\le 3t$,
we get $c(t+3-c) \le 3t$. Set $\psi (x) =x(t+3-x)$. The function $\psi$ is strictly increasing if $1\le x \le (t+3)/2$ and strictly decreasing if $x> (t+3)/2$. Since $\psi (t) =3t$
and $\psi (3) =3t$, we get that either $1\le c\le 3$ or $c=t$.

\quad (a) Assume $c=1$. Since $Z$ spans $M$, we have $e_1\le \deg (Z)-1$. Since $e_1\le \deg (Z)-1$, we have $e_1<3t$. By \cite[Corollaire 2]{ep} (see also \cite[Remarque (i)]{ep}) 
either there is a line $R\subset N_1$ with $\deg (R\cap Z)\ge t+2$ or there is a conic $D\subset N_1$ such that $\deg (D\cap Z) \ge 2t+2$.

\quad (b) Assume $c=2$. Since $e_1\ge e_2$, we have $e_2\le \deg (Z)/2 \le 3t/2$. Since $c=2$ and $h^1(N_2,\Ii _{Z_1 \cap N_2,N_2}(t-1)) >0$, by \cite[Lemma 34]{bgi} either
$e_2 \ge 2t$, which is a contradiction, or there is a line $R\subset N_2$ such that $\deg (R\cap Z_1) \ge t+1$.
If $\deg (R\cap Z)\ge t+2$, then we are done. Hence we may assume $\deg (Z\cap R)=t+1$. Set $W_0:= Z$. Let $M_1\subset M$ be a plane containing $R$ and for
which $f_1:= \deg (M_1\cap Z)$ is maximal. Since $Z$ spans $M$ we have $f_1\ge t+2$. Set $W_1:= \Res _{M_1}(Z)$.  For each integer
$i\ge 2$ define recursively the plane $M_i$, the integer $f_i$ and the zero-dimensional scheme $W_i$ in the following way. Let $M_i\subset M$ be any hyperplane
such that $f_i:= \deg (W_{i-1}\cap M_i)$ is maximal. Set $W_i:= \Res _{N_i}(W_{i-1})$. We have $f_i\ge f_{i+1}$ for all $i\ge 2$, but we do not claim that
$f_1\ge f_2$ (indeed, $M_1$ is required to contain $R$, while the $M_i$ with $i \ge 2$ are not). Since any degree $3$ subscheme of $M$
is contained in a plane, if $f_i\le 2$, then $W_{i-1}\subset M_1$ and $W_i=\emptyset$. Since $\sum _{i\ge i} f_i =\deg (Z)$ and $f_1\ge t+2$, we have $f_i=0$ for some $i<t$.
Using the residual exact sequences of the planes $M_i$ we get the existence of a minimal integer $s\in \{1,\dots ,t-1\} $ such that $h^1(M_s,\Ii _{W_{s-1}\cap M_s,M_s}(t+1-s)) >0$.
We get $f_s\ge t+3-s$. Since $f_1\ge t+2$, we get $1\le s \le 2$. If $s=1$, then we use step (a) with $M_1$ instead of $N_1$. Now assume $s=2$. Since $f_2\le \deg (Z) -f_1\le 2t-2$ and $h^1(M_2,\Ii _{Z_1\cap M_2,M_2}(t-1)) =0$, there is a line $L\subset M_2$ such that $\deg (L\cap Z_1)\ge t+1$. If $\deg (L\cap Z)\ge t+2$, then the lemma is true.
Hence we may assume that $\deg (Z\cap L)=t+1$. 

First assume $R\cap L =\emptyset$. Let $Q\subset M$ be a general quadric surface containing $L\cup R$. Call $|\Oo _Q(1,0)|$ the ruling of $Q$ containing $R$ and $L$.
The residual scheme $\Res _Q(Z)$ of $Z$ has degree
$\deg (Z) -\deg (Z\cap Q) \le 3t -(2t+2) =t-2$ and in particular $h^1(M,\Ii _{\Res _Q(Z),M}(t-2)) =0$. The residual exact sequence of $Q$ gives $h^1(Q,\Ii _{Z\cap Q,Q}(t))
\ge h^1(M,\Ii _{Z,M}(t)) >0$. 

\quad {\emph{Claim:}} We have $h^1(M,\Ii _{R\cup L}(t)) =0$ for every $t\ge 1$.

\quad {\emph{Proof of the Claim:}} Take $p\in L$. Since $R\cap L=\emptyset$, $\{p\}\cup R$ spans a plane, $H$. We have $(L\cup R)\cap H =R\cup \{p\}$ and hence $h^1(H,\Ii _{(L\cup R)\cap H,H}(t)) =0$ for all $t\ge 1$. The residual $\mathrm{Res}_H(L\cup R)$ of $L\cup R$ with respect to $H$ is the line $L$, because $L\cup R$ is reduced and $L\nsubseteq H$. Therefore the residual sequence of $H$ in $\PP^3$ gives the following exact sequence:
$$
0\to \Ii _L(t-1) \to \Ii _{L\cup R}(t)\to \Ii _{(L\cup R)\cap H,H}(t)\to 0.
$$
Since $h^1(\Ii _L(t-1)) =0$ for all $t>0$, we get the Claim.

Since $\deg (Z\cap L) = \deg (Z\cap R) =t+1$ and $R\cap L=\emptyset$, we have $h^1(R\cup L,\Ii _{(R\cup L)\cap Z}(t)) =h^1(R,\Ii _{R\cap Z}(t)) +h^1(L,\Ii _{L\cap Z}(t)) =0$. 

The Claim gives $h^1(M,\Ii _{Z\cap (R\cup L}(t)) =0$. Hence $h^1(Q,\Ii _{Z\cap (R\cup L),Q}(t)) =0$.
 The residual sequence
$$
0 \to \Ii _{\Res _{R\cup L}(Z\cap Q)}(t-2,t)) \to \Ii _{Z\cap Q,Q}(t,t) \to \Ii _{(R\cup L)\cap Z,R\cup L}(t,t)\to0
$$
gives $h^1(Q,\Ii _{\Res _{R\cup L}(Z\cap Q)}(t-2,t)) >0$.
Since $\deg (\Res _{R\cup L}(Z\cap Q)) = \deg (Z\cap Q) -2t-2 \le t-2$, we get a contradiction.

Now assume $R\cap L \ne \emptyset$. If $R\ne L$, then we may take the reducible conic $R\cup L$, because $R\subset M_1$ and $\deg (L\cap \Res _{M_1}(Z)) =t+1$.

Now assume $R=L$. This is the last case of the statement of the lemma.

\quad ({c}) Assume $c=3$. Since $\psi (3) =3t$, we get $e_1=e_2= e_3 =t$. Since $e_3=t$ and $h^1(N_3,\Ii _{Z_2\cap N_3,N_3}(t-2)) >0$, there is a line
$R\subset N_3$ such that $\deg (Z_2\cap R) =t$. Since $Z$ spans $M$, there is a plane $N'$ such that $R\subset N'\subset M$ and $N'\cap Z\supsetneq Z\cap R$.
Hence $e_1\ge \deg (N'\cap Z)>t$, a contradiction.

\quad (d) Assume $c=t$. We get $\deg (Z)=3t$ and $e_i=3$ for all $i$. In particular $e_1=3$, i.e. $Z$ is in linearly general position. Since $\deg (Z) \le 3t+1$, the contradiction comes from \cite[Theorem 3.2]{eh}.
\end{proof}

\begin{lemma}\label{c2.1}
Fix an integer $t>0$. Set $M:= \PP^3$ and let $Z\subset M$ a  zero-dimensional and curvilinear scheme spanning $M$ and with $\deg (Z) \le 3t$. We have
$h^1(M,\Ii _{Z,M}(t))$ $\ne 0$ if and only if either there is a line $R\subset M$ with $\deg (R\cap Z)\ge t+2$ or there is a conic $D\subset M$
such that $\deg (D\cap Z)\ge 2t+2$.
\end{lemma}

\begin{proof}
The ``~if~'' part is trivial. To prove the other implication it is sufficient to exclude the last case of the statement of Lemma \ref{c1.1}. By \cite[Corollaire 2]{ep} (see also \cite[Remarque (i)]{ep}) 
we may assume that $h^1(N,\Ii _{Z\cap N,N}(t)) =0$ for every plane $N$. 

Assume that we are in the last case of Lemma \ref{c1.1} and call $L$ the associated line. We may take $Z$ minimal with the property that $h^1(M,\Ii _{Z,M}(t)) >0$. 
Let $Q$ be a quadric surface containing $L$ in its singular locus. Since $\deg (\Res _Q(Z)) \le 3t-2t-2 \le t-2$, we have $h^1(M,\Ii _{\Res _Q(Z)}(t-2)) =0$. Therefore
the residual exact sequence of $Q$ gives $h^1(Q,\Ii _{Z\cap Q,Q}(t)) >0$ and $h^1(M,\Ii _{Z\cap Q,M}(t)) >0$. The minimality of $Z$ gives $Z\subset Q$. Taking $Q =N_1+N_2$
in step (b) of the proof of Lemma \ref{c1.1} we also get that only the connected components of $Z$ whose reduction are contained in $L$ arise (for a minimal $Z$), 
hence we reduce to the case $\deg (Z)=2t+2$. 

Let $W\subset Z$ be any degree $2t+1$ subscheme.
Since $\deg (W\cap D)\le \deg (Z\cap D)\le t+1$ for each line $D$, Lemma \ref{c1.1} gives $h^1(M,\Ii _{W,M}(t)) =0$.
Hence $h^1(M,\Ii _{Z,M}(t)) =1$. Since $h^1(N,\Ii _{Z\cap N,N}(t)) =0$ for every plane $N$, as in \cite{da} we get $h^1(M,\Ii _{Z,M}(t-1)) \ge 3+h^1(M,\Ii _{Z,M}(t))=4$.
Let $N$ be any plane containing $L$. We have $h^1(N,\Ii _{Z\cap N}(t-1)) =1$, because $\deg (Z\cap L)=t+1$ and $\deg (Z\cap N)\le 2(t-1)+1$ (use the residual exact sequence
of $L$ in $N$). Since $\deg (\Res _N(Z)) \le t+1$, we have $h^1(M,\Ii _{\Res _N(Z)}(t-2)) \le 2$. Hence the residual exact sequence of $N$
gives $h^1(M,\Ii _{\Res _N(Z)}(t-1)) \le 2+1$, a contradiction.
\end{proof}

\begin{lemma}\label{bd2}
Let $H\subset \PP^4$ be a hyperplane. Let $S\subset H$ be a set of $12$ points in uniform position and spanning $H$.

\quad (a) $h^1(H,\Ii _{S,H}(3)) \ge 2$ if and only if $S$ is contained in a rational normal curve of $H$ and in this case we have $h^1(H,\Ii _{S,H}(3))=2$;

\quad (b) $h^1(H,\Ii _{S,H}(3))=1$ if and only if $S$ is contained in an integral curve $T\subset H$, which is the complete intersection of two quadric surfaces.
\end{lemma}

\begin{proof}
If $S$ is contained in a rational normal curve (resp. an integral complete intersection of two quadric surfaces), then $h^1(H,\Ii _{S,H}(3))=2$
(resp. $h^1(H,\Ii _{S,H}(3))=1$).
Since $S$ is in linearly general position, we have $h^1(H,\Ii _{S'}(3)) =0$ for each $S'\subset S$ with $\sharp (S') =10$. Hence $h^1(H,\Ii _{S,H}(3))\le 2$.
If $h^0(H,\Ii _{S,H}(2)) \ge 2$, since $S$ is in uniform position we get that $S$ is contained in a integral curve with either degree $3$ or the intersection
of $2$ quadric surfaces. Hence we may assume $h^0(H,\Ii _{S,H}(2)) \le 1$. There is $A\subset S$ with $\sharp (A) =8$ and $h^0(H,\Ii _{A,H}(2)) =2$, i.e.
$h^1(H,\Ii _{A,H}(2))=0$. Take an ordering $o_1,o_2,o_3,o_4$ of $S\setminus A$. Set $A_0:= A$.
For $i=1,2,3,4$ set $A_i:= A\cup \{o_1,\dots ,i\}$. It is sufficient to prove that $h^0(H,\Ii _{A_i,H}(3)) < h^0(H,\Ii _{A_{i-1},H}(3))$ for
$i=1,2,3,4$. Let $Q$ be a general quadric surface containing $A$. Since $S$ is in uniform position, we have $Q\cap S =A$. Let $N_i$ be any plane
not containing $o_i$ but containing $o_j$ for all $j<i$. The cubic surface $Q\cup A_i$ gives $h^0(H,\Ii _{A_i,H}(3)) < h^0(H,\Ii _{A_{i-1},H}(3))$.
\end{proof}

\begin{lemma}\label{m3}
Let $C\subset \PP^r$, $r\ge 2$, be a smooth rational curve. Let $M(d,r)$ denote the set of all smooth rational curves of degree $d$ in $\PP^r$.
$M(d,r)$ is smooth and irreducible of dimension $(r+1)d +r-3$.
Set $d:= \deg ({C})$ and take a zero-dimensional
scheme $Z\subset \PP^n$ such that $a:= \deg(Z) \le d+1$. Then $h^1(N_C(-Z)) =0$ and the set of all $X\in M(d,r)$ containing $Z$ has
dimension ($r+1)d+r-3 -(r-1)a$.
\end{lemma}

\begin{proof}
Fix any $Y\in M(d,r)$. Since $T\PP^r$ is a quotient of $\Oo _{\PP^r}(1)$ by the Euler sequence and $X$ is smooth $N_X$ is a quotient of $\Oo _X(1)^{(r+1)}$.
Since $X$ is a smooth rational curve, we get $h^1(N_X(-W)) =0$ for every zero-dimensional scheme $W\subset X$ with $\deg (Z) \le d+1$. The Hilbert scheme
of all curves containing $W$ has $H^0(N_X(-W))$ as its tangent space and $H^1(N_X(-W))$ as an obstruction space (\cite[Theorem 1.5]{pe}).
Taking $W=\emptyset$
we get the smoothness and dimension of $M(d,r)$. The irreducibility of $M(d,r)$ is well-known. Taking $W = Z$ we get the other statements of the lemma.
\end{proof}

Let $\Ww$ denote the set of all quintic hypersurfaces of Cotterill, i.e. satisfying all properties proved in \cite{c2}. In particular each $W\in \Ww$ is a smooth quintic hypersurface containing finitely many rational curves of degree $\le 11$.

For any integer $b\ge 5$ let $\Delta _b$ denote the set of all non-degenerate $C\in M_{12}$ such that there is a line $L\subset \PP^4$ with $\deg (L\cap C) =b$. Set $\Delta '_7:= \cup _{b\ge 7} \Delta _b$.

\begin{remark}\label{0a00+} For any line $L\subset \PP^4$ let $A(L,b)$ denote the set  of all non-degenerate 
$C\in M_{12}$ such that $\deg (L\cap C) =b$. Since $\Delta _b=\emptyset$ if $b>12$, we have $\dim (A(L,b)) = 61-2b$
by Lemma \ref{m3}. Now, if $W$ is any quintic threefold and $C \subset W$, then by Bezout also $L \subset W$ as soon as $b \ge 6$. Since on each $W\in \Ww$ there are finitely many lines, if $W$ contains only finitely many $C \in A(L,b)$ 
for any fixed line $L\subset \PP^4$, then $W$ contains only finitely many $C \in \Delta _b$ as well. Hence to prove
that a general $W\in \Ww$ contains only finitely many elements of $\Delta _b$, by (\ref{one}) and (\ref{two})  
it is sufficient to test the element $C\in \Delta _b$ with $h^1(\Ii _C(5)) \ge 2b+1$.
\end{remark}

\begin{lemma}\label{0a8}
A general $W\in \Ww$ contains only finitely many $C\in \Delta '_7$.
\end{lemma}

\begin{proof}
By Remark \ref{0a00+} it is sufficient to test the non-degenerate curves $C\in M_{12}$ such that $h^1(\Ii _C(5)) \ge 15$. Take a general hyperplane $H\in \PP^4$.
Since $C\cap H$ is in uniform position, Lemma \ref{c2.1} gives $h^1(H,\Ii _{C\cap H,H}(t)) =0$ for $t=4,5$. The exact sequence 
\begin{equation}\label{nn3}
0 \to \Ii _C(t-1) \to \Ii _C(t)\to \Ii _{C\cap H,H}(t)\to 0
\end{equation}
gives $h^1(\Ii _C(3)) \ge h^1(\Ii _C(4)) \ge h^1(\Ii _C(5))$.
By Lemma \ref{bd2} we have $h^1(\Ii _C(2)) \ge h^1(\Ii _C(3)) -2\ge 13$. Hence $h^0(\Ii _C(2)) \ge 3$, contradicting Lemma \ref{0a1}.
\end{proof}

\begin{lemma}\label{0a9}
A general $W\in \Ww$ contains only finitely many $C\in \Delta _6$.
\end{lemma}

\begin{proof}
By Remark \ref{0a00+} it is sufficient to test the non-degenerate curves $C\in M_{12}$ such that $h^1(\Ii _C(5)) \ge 13$. By Lemma \ref{0a8} we may assume
that $C\notin \Delta '_7$. By Lemmas \ref{nn1} and \ref{c1.1} we have $h^1(\Ii _C(4)) \ge 4+h^1(\Ii _C(5)) \ge 17$. Take a general hyperplane $H\in \PP^4$.
By Lemma \ref{c2.1} we have $h^1(H,\Ii _{C\cap H,H}(4)) =0$. The exact sequence (\ref{nn3}) gives $h^1(\Ii _C(3)) \ge h^1(\Ii _C(4)) $.
By Lemma \ref{bd2} we have $h^1(\Ii _C(2)) \ge h^1(\Ii _C(3)) -2\ge 15$. Hence $h^0(\Ii _C(2)) \ge 5$, contradicting Lemma \ref{0a1}.
\end{proof}

\begin{lemma}\label{0a11}
A general $W\in \Ww$ contains only finitely many non-degenerate $C\in M_{12}$ such that there is a conic $D\subset \PP^4$ with $\deg (D\cap C)\ge 10$ and if the conic is singular $C\cap D$ contains a curvilinear scheme of at least degree $10$.
\end{lemma}

\begin{proof}
A conic is either smooth or reducible or a double line. Lemmas \ref{0a8} and  \ref{0a9} handle the case in which $D$ is not a smooth conic and $\deg (D\cap C)\ge 11$. Assume the existence of a conic $D$ such that $b:= \deg (D\cap C)\ge 10$. Fix any $p\in C\setminus C\cap N$, where $N$ is the plane spanned by $D$, and let $M$ be the hyperplane spanned by $N\cup \{p\}$. Since $\deg (C\cap M) \ge b+1$, we have $b \le 11$. $\PP^4$ contains $\infty ^6$ planes and each plane contains $\infty ^5$ smooth conics
and $\infty ^4$ singular conic. Fix $b\in \{10,11\}$ and a conic $D$. Let $B(D,b)$ be the set of all non-degenerate $C\in M_{12}$ such that $\deg (D\cap C)= b$; if $b=10$ and $D$ is singular assume that $D\cap C$ is curvilinear. Since each conic contains $\infty ^b$ curvilinear subschemes of degree $b$,
Lemma \ref{m3} gives $\dim (B(D,b)) \le 61 -2b$. Varying $D$ we get that the set of all $C$ has codimension at least $9$ in $M_{12}$. Hence it is sufficient to test the curves $C$
with $h^1(\Ii _C(5)) \ge 10$. Since $C\notin \Delta '_7$, we have $h^1(\Ii _C(4)) \ge 14$. Moreover, if $h^1(\Ii _C(4)) =14$, then $\deg (D\cap C) \ge 10$ for finitely many conics $D_1,\dots ,D_s$. Let $N_i$ be the plane spanned by $D_i$. Fix a line $L\subset \PP^4$ such that $L\cap N_i=\emptyset$ for all $i$. Set $V:= H^0(\Ii _L(1))$
and take any $M\in |\Ii _L(1)|$. We have $N_i\nsubseteq M$. Since $M\cap C$ contains no line $R$ with $\deg (R\cap C)\ge 6$ and no conic $D$ with $\deg (D\cap C) \ge 10$,
we have $h^1(M,\Ii _{C\cap M,M}(4)) =0$ (Lemma \ref{c2.1}). Hence the bilinear map $H^1(\Ii _C(4)) ^\vee \times V\to H^1(\Ii _C(3))^\vee$ is non-degenerate in the second variable. By the bilinear lemma
we have $h^1(\Ii _C(3)) \ge h^1(\Ii _C(4)) +\dim (V)-1 = 16$. Hence in all cases we have $h^1(\Ii _C(3)) \ge 15$. By Lemma \ref{bd2} we have $h^1(\Ii _C(2)) \ge 13$, contradicting Lemma \ref{0a1}.
\end{proof}

Let $\Delta _1$ (resp. $\Delta _2$, $\Delta _3$) be the set of all non-degenerate $C\in M_{12}$ such
that for a general hyperplane $H\subset \PP^4$ the set $C\cap H$ is contained in a rational normal curve of $H$ 
(resp., the smooth complete intersection of $2$ quadric surfaces of $H$, resp., a singular integral curve which is 
the complete intersection of 2 quadric surfaces of $H$). 

We have the following estimates:

\begin{lemma}\label{co1}
Every irreducible component of $\Delta _1$ has dimension $\le 49$.
\end{lemma}

\begin{proof}
Fix a hyperplane $H$, a rational normal curve $D\subset H$ and $S\subset D$ such that $\sharp (S) =12$. By Lemma \ref{m3} the set of all $C\in M_{12}$
containing $S$ has dimension $\le 61-36 $. Since the set of all $S\subset D$ with $\sharp (S) =12$ has dimension $12$ 
and $H$ contains $\infty ^{12}$
rational normal curves, we get the lemma.
\end{proof}

\begin{lemma}\label{co2}
Every irreducible component of $\Delta _2$ has dimension $\le 53$.
\end{lemma}

\begin{proof}
Fix a hyperplane $H$. The set of all degree $4$ smooth elliptic curves of $H$ has dimension $16$ and we may conclude as in the proof of Lemma \ref{co1}.
\end{proof}

\begin{lemma}\label{co3}
Every irreducible component of $\Delta _3$ has dimension $\le 52$.
\end{lemma}

\begin{proof}
Fix a hyperplane $H$. The set of all singular, integral and non-degenere curves $D\subset H$ with $\deg (D\cap H) =4$, i.e. the set of all singular integral curves which are the complete intersection of 2 quadric surfaces of $H$, has dimension $15$. Now we argue as in the proof of Lemma \ref{co1}.
\end{proof}

\begin{lemma}\label{0a12}
A general $C\in M_{12}$ contains only finitely many elements of $\Delta _1 \cup \Delta _2 \cup \Delta _3$.
\end{lemma}

\begin{proof}
By Lemmas \ref{co1},  \ref{co2} and \ref{co3}, we may assume that $h^1(\Ii _C(5)) \ge 9$. 
By Lemmas \ref{0a8}, \ref{0a9} and \ref{0a11} we may assume $\deg (C\cap L)\le 5$ for all lines
and $\deg (D\cap C)\le 9$ for all conics. By Lemmas \ref{nn1} and \ref{c1.1} for $t=4,5$ we get $h^1(\Ii _C(3)) \ge 4+h^1(\Ii _C(4))\ge 8+h^1(\Ii _C(5)) \ge 17$.
Lemma \ref{bd2} gives $h^1(\Ii _C(2)) \ge 15$, i.e. $h^0(\Ii _C(2)) \ge 5$, contradicting Lemma \ref{0a1}.
\end{proof}

By Lemmas \ref{0a8}, \ref{0a9} and \ref{0a11} to prove Theorem \ref{clemens} for non-degenerate $C\in M_{12}$ it is sufficient to test the ones such that $\deg (C\cap D)\le 9$ for any conic $D$ and $\deg (L\cap C)\le 5$ for any line $L$. 
By the cases $t=4,5$ of Lemmas \ref{nn1} and \ref{c1.1} we have
$h^1(\Ii _C(3)) \ge 4+h^1(\Ii _C(4)) \ge 8+h^1(\Ii _C(5))$. 

By Lemmas \ref{bd2} and \ref{0a12} we may assume $h^1(H,\Ii _{C\cap H,H}(3))=0$. 
Now the case $t=3$ of the exact sequence (\ref{nn3})
gives 
\begin{equation}\label{chain}
h^1(\Ii _C(2)) \ge h^1(\Ii _C(3)) \ge 4+h^1(\Ii _C(4)) \ge 8+h^1(\Ii _C(5)).
\end{equation}
 
Since the stratum in $M_{12}$ corresponding to curves with $h^1(\Ii _C(5)) > 0$ has codimension $2$ 
(as in \cite[pp. 901--902]{da}), by (\ref{one}) and (\ref{two})  we may assume $h^1(\Ii _C(5)) \ge 3$, 
hence  $h^1(\Ii _C(2)) \ge 11$. Since $h^0(\Oo _{\PP^4}(2)) =15$ and $h^0(\Oo _C(2)) =25$, 
we get $h^0(\Ii _C(2)) \ge 1$. 

Now, if $h^1(\Ii _C(5)) \le 5$ (hence  $h^0(\Ii _C(5)) \le 70$) we conclude by the following Lemma \ref{fff1}. 

\begin{lemma}\label{fff1}
Let $\Gamma$ be any irreducible family of non-degenerate curves of $M_d$, $d >1$, contained in some quadric hypersurface. Then $\dim \Gamma \le 14+3d$.
\end{lemma}

\begin{proof}
Since $\dim |\Oo _{\PP^4}(2)| =14$ and singular quadrics occur in codimension $1$, it is sufficient to prove that 
for every smooth (resp., integral but singular) quadric $Q$ the set $\Gamma'$ of all $C\in M_d$ contained in $Q$
has dimension $\le 3d$ (resp., $\le 3d+1$). 

First assume that either $Q$ is smooth or $C$ does not intersect the singular locus $V$ of $Q$.
In this case the normal sheaf $N_{C,Q}$ is a rank $2$ spanned vector bundle on $C$, hence $h^1(N_{C,Q})=0$. Since $\det (N_{C,Q})$ has degree $3d-2$ and $N_C$ has rank $2$,
Riemann-Roch gives $h^0(N_{C,Q})=3d$, proving the lemma in this case. 

Now assume $C\cap V\ne \emptyset$ and set $x:= \deg (C\cap V)$. Since $C$ is smooth, $x=1$ if $\dim (V)=0$. Let $\tau _Q$ denote the tangent sheaf
of $Q$. The vector space $H^0(\tau _Q)$ is the tangent space at the identity map of the automorphism group $\mathrm{Aut}(Q)$. Since $Q\setminus V$
is homogeneous, $\tau _Q|(Q\setminus V)$ is a spanned vector bundle. Since $C$ is not a line
and $\dim V\le 1$, the set $V\cap C$ is finite. Dualizing the natural map from the conormal sheaf of $C$ in $Q$ to $\Omega _Q^1$ we
get a map $u: \tau _Q|C \to N_{C,Q}$ which is surjective outside the finite set $C\setminus C\cap V$. Since $C$ is smooth and rational
and $\tau _Q$ is spanned at each point of $Q\setminus V$, we get $h^1(N_{C,Q})=0$. Since we need to prove that $\dim \Gamma ' \le 3d+1$, it is sufficient to check this inequality when $C$ is a general
element of $\Gamma '$. In particular we may assume that $\deg (C'\cap V) =x$ for a general $C'\in \Gamma '$ and use induction on the integer $x$, the case $x=0$ being true
by the case $C\cap V=\emptyset$ proved before. Set $\Gamma '':= \{C'\in \Gamma ': \deg (V\cap C) =x\}$. It is sufficient to prove that $\dim \Gamma '' \le 3d+1$. 
Let $v: \widetilde{Q} \to Q$ be the blowing up of $V$, $E:= v^{-1}(V)$ the exceptional divisor, and $\widetilde{C} \subset \widetilde{Q}$ the strict transform of $C$. Since $C$ is smooth, $v$ maps isomorphically
$\widetilde{C}$ onto $C$ and the numerical class of $\widetilde{C}$ with respect to $\mathrm{Pic}(\widetilde{Q})$ only depends on $\dim (V)$, $d$ and $x$. Let $\Psi$ be closure in $\mathrm{Hilb}(\widetilde{Q})$ of the strict transforms of all $C'\in \Gamma ''$. It is sufficient to prove
that $\dim \Psi \le 3d+1$. Take a general $D\in \Psi$. Since $\mathrm {Aut}(\widetilde{Q})$ acts transitively of $\widetilde{Q}\setminus E$, the first part of the proof
gives $h^1(N_{D,\widetilde{Q}}) =0$. Hence it is sufficient to prove that $\deg (N_{D,\widetilde{Q}}) \le 3d-1$, i.e. $\deg ({\tau _{\widetilde{Q}}}_{|D}) \le 3d+1$, i.e. $\deg (\omega _{\widetilde{Q}}|D) \ge -3d-1$. 
The group $\mathrm{Pic}(\widetilde{Q})$ is freely generated by $E$ and the pull-back $H$ of $\Oo _Q(1)$. We have $D\cdot H = d$ and $D\cdot E =x$.
We have $\omega _{\widetilde{Q}}\cong \Oo _{\widetilde{Q}}(-3H+cE)$ with $c = -1$ if $\dim(V)=0$ (see for instance 
\cite{jpr}, Example 8.5 (2)) and  $c = 0$ if $\dim(V)=1$ (see for instance \cite{jpr}, Example 8.5 (3)).
Hence $\deg ({\omega _{\widetilde{Q}}}_{|D}) = -3d+cx \ge -3d-1$ and the proof is complete.
\end{proof}

If instead $h^1(\Ii _C(5)) \ge 6$, then by (\ref{chain}) we have  $h^1(\Ii _C(2)) \ge 14$, i.e. $h^0(\Ii _C(2)) \ge 4$,
contradicting  Lemma \ref{0a1}.

\section{Degenerate case}\label{deg}

The degenerate case occurs in codimension $10$ of $M_{12}$. Indeed, the general curve of degree $d=12$ in $\PP^3$ 
has maximal rank (\cite{be2}), in particular it does not sit on any quintic. It follows that our codimension is $\dim(M_d)-(4d-1+4)=61-51=10$. Hence we may assume  $h^1(\Ii _C(5)) \ge 11$.

We consider degenerate curves $C\in M_{12}$ with $h^1(\Ii _C(5)) \ge 11$ contained in a hyperplane $M$ and in a general quintic $W$ with $W':= M\cap W$. 

\begin{lemma}\label{deg1}
Let $W \subset \PP^4$ be a general quintic hypersurface. Then $W$ contains finitely many integral curves $T$ of degree $4$ which are the complete intersection
of a hyperplane and $2$ quadric hypersurfaces and all of them are smooth.
\end{lemma}

\begin{proof}
Since $W$ contains no singular rational curves (\cite{c2}), it is sufficient to consider the smooth ones, i.e. the degree $4$ elliptic curves of $\PP^4$. Let $\Gamma '$ be the set of all 
degree $4$ elliptic curves of $\PP^4$. Fix $T\in \Gamma '$. Since $N_T \cong \Oo _T(2)^{\oplus 2} \oplus \Oo _T(1)$, we have $h^1(N_T)=0$, hence $\Gamma'$ is smooth and of dimension $\chi (N_T) = 16$. Since $T$ is a complete intersection, we have
$h^1(\Ii _T(5)) = 0$ and $h^0(\Ii _T(5)) = \binom{9}{4} -20$. Hence a dimension count gives the lemma.
\end{proof}

\begin{lemma}\label{non1}
$W'$ contains only finitely many non-degenerate curves of degree $5$ and $6$.
\end{lemma}

\begin{proof}
Fix a degree $t\in \{5,6\}$ integral and non-degenerate curve $D\subset W'$ and set $q:= p_a(D)$.  By \cite{glp} we have $h^1(M,\Ii _{D,M}(5)) =0$, hence $h^0(\PP^4,\Ii _D(5)) =126-5t-1+q$. 

First assume $t=5$. By the genus bound for space curves we have $q\le 2$. Since $q\le 2$, we have
$h^1(\Oo _D(1)) =0$ and in particular $h^1(\Oo _D(5)) =0$, i.e. $h^0(\Oo _D(5)) =5t+1-q$. Since $q\le 2$, all the irreducible components
of the Hilbert scheme of $M$ containing $D$ have dimension $20$. Since $\PP^4$ has $\infty ^4$ hyperplanes, it is sufficient to use that $4t +4
\le 5t+1-q$.

Now assume $t=6$. By the genus bound for space curves (\cite[Theorems 3.7 and 3.13]{he}), $q\le 3$ and $q=3$ if and only if $D$ is contained
in a quadric surface $Q$. Assume $q=3$. In this case $D$ is the complete intersection of $Q$ and a cubic surface (\cite[Corollary 3.14]{he})
and so $D$ is a locally complete intersection, $\omega _D \cong \Oo _D(1)$ and $N_{D,M} \cong \Oo _D(2)\oplus \Oo _D(3)$. Since $h^1(N_D)=0$
the Hilbert scheme of $M$ at $D$ is smooth and of dimension $4t$. We conclude as in the case $t=5$.
The case $q\le 2$ is done as in the case $t=5$.
\end{proof}

A theorem of Zak (see for instance \cite{z}) states that the Gauss map of any smooth projective variety is finite, hence $W'$ has only finitely many singular points, all of them being hypersurface singularities. 
By \cite[p. 733]{msv} $W'$ has only rational double points of type $A_i$, $i\le 4$, and $D_4$ as singularities.

We may improve the lower bound $h^1(\Ii _C(5)) \ge 11$ if we restrict the set of hyperplanes or rather if we restrict the pairs $(W,M)\in |\Oo _{\PP^4}(5)|\times |\Oo _{\PP^4}(1)|$.

\begin{remark}\label{wo1}
If $M$ is tangent to $W$, i.e. if $W'$ is singular, then we may assume $h^1(\Ii _C(5)) \ge 12$. Since the Gauss map is birational, if $W'$ has at least two singular
points, then we may assume $h^1(\Ii _C(5)) \ge 13$.
\end{remark}

\begin{remark}\label{wo2}
For any line $L\subset \PP^4$ we have $h^0(\Ii _L(1)) =3$. A general $W$ contains only finitely many lines (\cite{c2}). Hence if $W'$ contains a line, then we may assume
$h^1(\Ii _C(5)) \ge 13$. Since any two lines of $W$ are disjoint (\cite{c2}), any two lines of $W$ span a hyperplane.
Hence if $W'$ contains two different lines, then we may assume $h^1(\Ii _C(5)) \ge 15$. Fix a line $L\subset W$. For any $p\in L$, the hyperplane $T_pW$
is the only hyperplane singular at $p$. Since $\dim (L) =1$, we get that if $W'$ is singular at one point of $L$, then we may assume $h^1(\Ii _C(5)) \ge 14$.
\end{remark}

\begin{remark}\label{wo3}
For any smooth conic $D\subset \PP^4$ we have $h^0(\Ii _D(1)) =2$. A general $W$ contains only finitely many conics (\cite{c2}). Hence if $W'$ contains a smooth conic, then we may assume
$h^1(\Ii _C(5)) \ge 14$.
\end{remark}

\begin{remark}\label{wo4}
For any integer $x$ with $3\le x\le 11$, $W$ contains only finitely many smooth rational curve of degree $x$, none
of them contained in a plane. Hence if $W'$ contains a smooth rational curve
of degree $x$, then we may assume $h^1(\Ii _C(5)) \ge 15$. The same is true if $W'$ contains a line and a conic or $2$ conics.
\end{remark}

For any hyperplane $U$ let $M_{12}(U)$ denote the set of all $C\in M_{12}$ contained in $U$. 
The locus $M_{12}(U)$ is smooth and irreducible and $\dim (M_{12}(U)) =48$.

\begin{remark}\label{wo5}
Fix an integer $e>0$ and assume the existence of a line $L\subset W'$ such that $\deg (L\cap C)=e$. Let $\Jj (e)$ be the set of all quadruples $(W,H,L,C)$ with $W\in \Ww$, $H$
a hyperplane, $L\subset W\cap U$ a line, $C\in M_{12}(U)$ and $\deg (L\cap C)=e$. Fix any $(W,H,L,C)\in \Jj (e)$. We have $\Jj (e) =\emptyset$ if $e\ge 12$. Now assume $e\le 11$.
Fix a line $L\subset M$ and a degree $e$ zero-dimensional
scheme $Z\subset U$ with $\deg (Z) =e$ and take any $C\in M_{12}(U)$ such that $Z\subset C$. As in Lemma \ref{m3} we see that $h^1(N_{C,M}(-Z)) =0$, hence the set of
all $C\in M_{12}(U)$ with $Z\subset C$ has dimension $48-2e$. Varying $Z$ in $L$ we see that the set of all $C\in M_{12}(U)$ the set of all $C\in M_{12}(U)$ such that
$\deg (C\cap L) =e$ has dimension $\le 48-e$.
Since each $W\in \Ww$ contains only finitely many lines, to show that for all $(W,M,L,C)\in \Jj (e)$ we have $C \nsubseteq W$ it is sufficient to exclude the ones with $h^1(\Ii _C(5)) \ge 13+e$. 
\end{remark}

Since the Gauss map of the smooth projective variety $W$ is finite, $W'$ has only finitely many singular points. Since $W'$ is locally a complete intersection, $W'$ is normal. By \cite{c2} $W$ has only finitely many lines
and only finitely many conics and no singular rational curve of degree $\le 11$. 
By Lemma \ref{deg1} $W$ has only finitely many smooth elliptic curves of degree $4$.

\begin{remark}\label{aaa1}
Let $W \subset \PP^4$ be a general quintic. Let $\Dd _i$, $i\ge 1$, be the set of all irreducible plane curves of degree $i$ contained in $W$. Since $W$ contains no plane, we have $\Dd _i =\emptyset$ for all $i\ge 6$,
and the set $\Dd _5$ is formed by the irreducible degree $5$ curves of the form $W\cap N$ with $N\subset \PP^4$ a plane. Hence $\Dd _5$ is irreducible and of dimension $8$.
By \cite{c2},  $\Dd _1\cup \Dd _2$ is finite and any two elements of it are disjoint. Fix $D\in \Dd _3$ and let $N \subset \PP^4$ be the plane spanned by $D$. The plane curve $W\cap N$ is the union of 
some $D\in \Dd _3$ together with a smooth conic, a reducible conic, or a double line. Since $\Dd _2$ is finite, the first case may occur only for finitely many planes and these are exactly the planes $N$ 
such that $W\cap N = T_2\cup T_3$ with $T_2\in \Dd _2$ and $T_3\in \Dd _3$. The second case does not occur, because the lines of $W$ are disjoint. Now assume
that $W\cap N = D\cup 2L$ with $L$ a line. By Zak's tangency theorem the restriction to $L$
of the Gauss map of $W$ is finite. Therefore the third case occurs only for at most one plane $N\supset L$
Now take $T\in \Dd _4$ and let $N$ be the plane spanned by $T$. We have $N\cap W = T\cup R$ with $R\in \Dd _1$, hence all elements of $\Dd _4$ are obtained
in the following way. 
Since the set of all planes of $\PP^4$ containing a line is a $2$-dimensional projective space, each irreducible component of $\Dd _4$ has dimension $2$.
 Fix any $L\in \Dd _1$ and take the intersection with $W$ of the element of the net of all planes of $\PP^4$ containing $L$. For a fixed hyperplane $M\subset L$
the set of all planes containing $L$ and contained in $M$ has dimension $1$.
\end{remark}

Let $\alpha$ be the minimal degree of a surface of $M$ containing $C$. Since $C$ is irreducible, every degree $\alpha$ surface containing $C$ is irreducible.

\begin{lemma}\label{wo6}
We have $\alpha >3$.
\end{lemma}

\begin{proof}
Since $C$ spans $M$, we have $\alpha >1$. Assume $\alpha =2$ and take $Q\in |\Ii _{C,M}(2)|$. Since $W'$ is irreducible, $W'\cap Q$ is a degree $10$ curve
containing $C$, a contradiction.
Now assume $\alpha =3$. Since $\deg ({C}) >9$, $C$ is contained in a unique cubic surface $S$. Let $J\subset S\cap W'$ be the locally Cohen-Macaulay curve
linked to $C$ by the complete intersection $S\cap W'$. We have $\deg (J) =3$ and $p_a(J) =-18$ (\cite[Proposition 3.1]{ps}).
Since $\deg (J) < -p_a(J)$, $J$ has a multiple component. Since $\deg (J)=3$, the multiple component is a line, $L$.
Since $|\Ii _{C,M}(5)|$ contains all quintic surfaces $S\cup Q$ with $Q\in |\Oo _M(2)|$ and $W'$ is irreducible, we have
$h^0(M,\Ii _{C,M}(5)) \ge 11$, i.e. $h^1(M,\Ii _{C,M}(5)) \ge 16$. Assume for the moment the non-existence of a line $R\subset M$ with $\deg (R\cap C)\ge 7$.
By Lemmas \ref{nn1} and \ref{c1.1} we get $h^1(M,\Ii _{C,M}(4)) \ge 19$. Fix a general plane $N\subset M$. We have an exact sequence
\begin{equation}\label{eqwo1}
0 \to \Ii _{C,M}(t-1) \to \Ii _{C,M}(t)\to \Ii _{C\cap N,N}(t)\to 0
\end{equation}
Since $N$ is general, the plane cubic $C\cap N$ is irreducible and $C\cap N$ is formed by $12$ points of the smooth locus of $C\cap N$.
Hence $h^1(N,\Ii _{C\cap N,N}(4)) \le 1$ with equality if and only if $C\cap N$ is the complete intersection of $S\cap N$ with a plane quartic. Since $h^0(M,\Ii _{C,M}(2)) =0$
and (by the genus formula) $C$ is not a complete intersection of two surfaces, \cite[Theorem 6]{st} gives
$h^1(N,\Ii _{C\cap N,N}(4)) =0$. The case $t=4$ of (\ref{eqwo1}) gives $h^1(M,\Ii _{C, M}(3)) \ge 19$, i.e. $h^0(M,\Ii _{C,M}(3)) \ge 2$, a contradiction.
Now assume the existence of a line $R\subset M$ such that $e:= \deg (R\cap C)\ge 7$. There are at most finitely many such $R$, because they cannot be all the lines of a ruling of $S$. Take a line $L\subset M$ disjoint from all $R$. Set $V:= H^0(M,\Ii _{L,M}(1))$. Take any plane $U\subset M$ containing $L$. Since $\deg (K\cap C)\le 6$ for each line $K\subset U$, we have $h^1(U,\Ii _{C\cap U,U}(5)) =0$. Hence the bilinear map $H^1(M,\Ii _{C,M}(5)) ^\vee \times V\to H^1(\Ii _{C,M}(4)) ^\vee$ is non-degenerate. Since $\dim (V)=2$, the bilinear lemma
gives $h^1(M,\Ii _{C,M}(4)) \ge h^1(M,\Ii _{C,M}(5)) +2-1$. Since $e>5$, Bezout gives $R\subset W'$. By Remark \ref{wo5} we may assume $h^1(M,\Ii _{C,M}(5)) \ge 20$. Since we just proved that $h^1(N,\Ii _{C\cap N,N}(4)) =0$, we get $h^1(M,\Ii _{C,M}(3)) \ge 21$, hence the contradiction $h^0(M,\Ii _{C,M}(3)) \ge 3$.
\end{proof}

\begin{lemma}\label{ww0}
$W'$ contains no $C\in M_{12}(M)$ such that $h^0(M,\Ii _{C,M}(4)) \ge 3$ and no $C\in M_{12}(M)$ with a line $L\subset M$ with $\deg (L\cap C)\ge 7$.
\end{lemma}

\begin{proof} The statement is made of two parts. 

\quad (a) Take general $S_1,S_2\in |\Ii _{C,M}(4)|$ and take a  general $(S_1,S_2)\in |\Ii _{C,M}(4)|^2$. Since $\alpha >3$, $S_i$ is irreducible.
The complete intersection $S_1\cap S_2$ links $C$ to a degree $4$ curve $J$ with $p_a(J) = -16$ (\cite[Proposition 3.1]{ps}), hence $J$ has at least one multiple
component, say $B$ with multiplicity $c\ge 2$. Since $h^0(M,\Ii _{C,M}(4)) \ge 3$, $J$ has also a movable component $A$. Hence $B$ is a line and either $A$ is a line or
it is a smooth conic. 

First assume that $A$ is a smooth conic, $c=2$, and $J$ has no other component. We have $p_a(C\cup B) =\deg (C\cap B) -1 \le 10$
and $p_a(A\cup B) = \deg (A\cap B)-1 \ge -1$. Since
$A\cup B$ is linked to $C\cup A$ by the complete intersection $S_1\cap S_2$, we have $p_a(A\cup B) = p_a(C\cup B)  -20 \le -10$, a contradiction.

Now assume $\deg (A)=1$. Moving $S_2$ we get that $S_1$ is ruled by lines. Since $\deg (S_1)>2$, $S_1$ has a unique ruling. This case cannot occur if $h^0(M,\Ii _{C,M}(4)) \ge 4$, because the plane is the only surface with $\infty ^2$ lines.

First assume that $c=2$. In this case
$J$ contains a line $R\notin \{B,A\}$. We have $p_a(C\cup B ) \le 10$, $p_a(B\cup A\cup R) \ge -2$, while \cite[Proposition 3.1]{ps} gives $p_a(B\cup A\cup R) = p_a(C\cup B)-20$,
a contradiction.
Now assume $c=3$. $C$, $A$ and $B$ are the unique components of $S_1\cap S_2$. $S_i$ and $S_2$ do not contain $B$ in their singular locus, because $S_1\cap S_2$ would contain $B$ with multiplicity $2$. Since the line
$B$ is not a line of the ruling of $S_1$, $S_1$ is not a cone, it is rational and it is a linear projection from a minimal degree surface $S\subset \PP^5$ (neither the Veronese
surface not a cone). $S$ is a Hirzebruch surface, either $F_0 \cong \PP^1\times \PP^1$ embedded by the complete linear system $|\Oo _{F_0}(h+2f)|$
or $F_2$ embedded by the complete linear system $|h+3f|$. $S_1$ is not a linear projection of $F_0$, because it has a line, $B$, not in the ruling and not in the singular locus
(i.e. the image of a conic of $F_0$). Hence $S_1$ is a linear projection of $F_0$. Any smooth rational curve $C_1\subset F_0$ with $C_1$ not a line
is an element of $|h+xf|$ for some $x\ge 3$. We have $\deg (C_1) = (h+xf)\cdot (h+3f) =x$, hence if $C_1$ has $C$ as its projection, then $x=12$. $B$ is the image
of $h$. We have $\deg (h\cap C_1) = 10$ and so $\deg (C\cap B) \ge 7$. Hence to prove the lemma it is sufficient to prove
the second assertion.

\quad (b) Take $C\in M_{12}(M)$ with a line $L\subset M$ with $\deg (L\cap C)\ge 7$. By part (a) to get a contradiction it is
sufficient to prove that $h^0(M,\Ii _{C,M}(4)) \ge 4$. Bezout gives $ R\subset W'$. By Remark \ref{wo5} we may
assume $h^1(M,\Ii _{C,M}(5)) \ge 19$. As in the proof of Lemma \ref{wo6} we get $h^1(M,\Ii _{C,M}(4)) \ge 20$, i.e. $h^0(M,\Ii _{C,M}(4)) \ge 6$.
\end{proof}

\begin{remark}\label{ww01}
Take $C\in M_{12}(M)$ without lines $L$ with $\deg (L\cap C)\ge 7$. By Lemma \ref{nn1} to prove that $h^0(M,\Ii _{C,M}(4)) \ge 3$ (hence to prove
that $C\nsubseteq W'$ by Lemma \ref{ww0}) it is sufficient to prove that $h^1(\Ii _C(5)) \ge 14$. By Remarks \ref{wo2}, \ref{wo3} and \ref{wo4} this is always the case if $W'$ contains
a smooth rational curve of degree $\ge 2$ or if it contains two lines. So from now on we assume that $W'$ has no such curves, 
hence no smooth elliptic curve
of degree $3$ by Remark \ref{aaa1}. We also assume that $W'$ has no smooth elliptic curve of degree $4$ by Lemma \ref{deg1}.
\end{remark}

Now we are going to apply all of the dimension-counting remarks and lemmas above and to use liaison in order to show that degenerate rational curves which are sufficiently generic 
(with respect to the properties described in the remarks and lemmas) must in fact have $h^1(\Ii _C(5)) < 11$, contradiction. Our argument hinges on a careful case-by-case analysis 
involving the types of divisors that that arise as components of certain residuals $C_T$ to $C$ inside of complete intersections of type $(5,5)$.

Since $h^1(\Ii _C(5)) \ge 11$ we have $h^0(M,\Ii _C(5)) \ge 6$. Hence $h^0(W',\Ii _{C,W'}(5)) \ge 5$. 
For any $T\in |\Ii _{C,M}(5))|$ with $T\ne W'$ let $C_T\subset T\cap W'$ denote the curve
linked to $C$ by the complete intersection $T\cap W'$. 

We have $\deg (C_T) =13$, $\chi (\Oo _{C_T}) =-2$ (\cite[Proposition 3.1]{ps}) and $h^1(\Ii _{C_T}(1)) = h^1(\Ii _C(5))$ (\cite[Theorem 1.1 (a)]{m1}, \cite{r}). 
Since  $\deg (C_T) =13$, $\chi (\Oo _{C_T}) =-2$, $C_T$ is not a plane curve (i.e. $h^0(M,\Ii _{C_T}(1)) =0$),
hence $h^0(\Oo _{C_T}(1)) =h^1(\Ii _{C_T}(1))+4=h^1(\Ii _C(5)) +4 \ge 15$. 
Since $\deg (C_T) >2p_a(C_T)-2$, we see that $C_T$ is not integral. 

Varying $T$ we find inside $W'$ a positive dimensional family of effective divisors $C_T$,
all of them linked to $C$ and with the same arithmetic genus, hence a flat family of effective divisors of $W'$. 
Therefore some of the effective divisors whose sum gives $C_T$ moves in $W'$. 

Let $D_1,\dots ,D_k$ be all all movable divisors of $C_T$
and let $R_1,\dots ,R_u$ the fixed divisors with multiplicities $e_1,\dots ,e_u$. Hence for a general $T$
we have $C_T=D_1+\cdots +d_k+e_1R_1+\cdots e_uR_u$ as effective Weil divisors of $W'$. 

Let $m(D_i)$, $1\le i\le k$, be the dimension of $D_i$ in the family $|\Ii _{C,W'}(5))|$.
We have $m(D_1)+\cdots +m(D_k) \ge 4$. We also proved that $m(D_1)+\cdots +m(D_k) \ge h^1(\Ii _C(5)) -7$. We saw that if $\deg (D_i) = 4$, then $W'$ contains a line $L$ and $m(D_i) =1$, 
because the moving family is induced by the family of planes of $M$ containing $L$. We saw that if $\deg (D_i) = 5$, then $D_i$ is a plane section of $W'$, hence $m(D_i)=3$. 

Let $R_1,\dots ,R_u$, $u\ge 0$, be the fixed divisors of $|\Ii _{W',C}(5)|$ and call $b_i$ the multiplicity of $R_i$ in $C_T$ (for a general $T$).

\quad (a) Assume the existence of $i\in \{1,\dots ,k\}$ such that $D_i$ is a plane curve of degree $5$. With no loss of generality we may assume $i=1$. Let $N$ be the plane containing $D_1$. Since $m(D_1) =3$, we have $k\ge 2$
and $h^0(W',\Ii _{C\cup D_1,W'}(5)) \ge 2$. Since $h^0(W',\Ii _{C\cup D_1,W'}(5)) = h^0(W',\Ii _{C,W'}(4))$.
Since $h^1(M,\Ii _{W',M}(4)) =h^1(\Oo _M(-1)) =0$,
we get $h^0(M,\Ii _C(4)) \ge 2$. Since $\alpha >3$ (Lemma \ref{wo6}), there are integral quartic surfaces $T_i\in |\Ii _{C,M}(4)|$, $i=1,2$,
with $T_1\ne T_2$. The complete intersection $T_1\cap T_2$ links $C$ to a locally Cohen-Macaulay curve $G$ such that $\deg (G) =4$.
and $p_a(G) =-16$ (\cite[Proposition 3.1]{ps}).
Since $p_a(G) \le -4$ and $\deg (G)=4$, $G$ has at least one multiple component, $J$, with multiplicity $e \ge 2$.  Our proof of the existence of $T_1$ and $T_2$
shows that we may take $T_1,T_2$ such that $D_2$ is a subcurve of $G$. Since $\deg (D_2)\ge 4 =\deg (G)$ and $G$ has a multiple component, we get a contradiction.

\quad (b) Assume the existence of $i\in \{1,\dots ,k\}$ such that $D_i$ is a plane curve of degree $4$. Just to fix the notation we assume $i=1$. Let $N$ be the plane spanned
by $D_1$. We have $W'\cap N = D_1\cup L$ with $L$ a line. Remark \ref{wo2} gives $h^0(\Ii _C(5)) \ge 7$, 
hence $m(D_1)+\cdots +m(D_k) \ge 6$. We saw that $m(D_1) =1$. Since $m(D_2)+\cdots +m(D_k) \ge 5$, we have $h^0(W',\Ii _{C\cup D_1,W'}(5)) \ge 6$. 
Since $\deg (L\cap D_1) =4$, we get $h^0(W',\Ii _{C\cup D_1\cup L}(5)) \ge 4$. Hence
we may find a movable divisor $E$ in $|\Ii _{C\cup L\cup D_1}(5)|$. We saw that $\deg (E)\ge 4$.
As in step (a) we get $h^0(M,\Ii _{C,M}(4)) \ge 4$, contradicting Lemma \ref{ww0}.

\quad ({c}) From now on we assume that each $D_i$ is non-degenerate. By Lemma \ref{non1} we may
assume $\deg (D_i) \ge 7$ for all $i$. By Remark \ref{ww01} we cannot have $2\le \deg (R_i) \le 4$ and we have $\deg (R_i)=1$
at most one index $i$. 

 Recall that $13 =\sum _{i=1}^{k} \deg (D_i) +\sum _{i=1}^{u} b_iR_u$ and we proved that $k+u>1$. Since $\deg (D_i)\ge 7$ for all $i$, we have $k=1$.
 
 Assume that $C_T$ has no multiple component. We have $h^0(\Oo _{D_1}(1)) +h^0(\Oo _{R_1}(1))+\cdots +h^0(\Oo _{R_u}(1)) \ge 2+h^1(\Ii _C(5))$.
Since $D_1$ moves, we have $p_a(D_1) >0$ (\cite{c2}), hence $h^0(\Oo _{D_1}(1)) \le \deg (D_1)$. Since $h^0(\Oo _{R_i}(1)) = \deg (R_i) +1$ for
at most one index $i$, we get a contradiction.

Hence  $C_T$ has at least one multiple component, say $R_1$. Since $\deg (D_1) \ge 6$, we get $b_1\deg (R_1) \le 13 -\deg (D_1) \le 7$ and
in particular $\deg (R_1)\le 3$. Since $W'$ has no curve of degree $2$ or $3$ (Remark \ref{ww01}), $R_1$ is a line, hence we may assume
$h^1(\Ii _C(5)) \ge 13$. Set $b:= b_1$, $R:= R_1$ and $e:= \deg (C\cap R)$. We have $\deg (D_1)=13-b$.
By Remarks \ref{wo5} and \ref{ww01} we may assume $e=0$, i.e. $R\cap C =\emptyset$ and that $R$ is contained in the smooth locus of $W'$.

\quad (d) Recall that $h^1(M,\Ii _{C,M}(5)) \ge 13$. By Lemma \ref{ww0} we may assume that $L$ has no line $T$ with $\deg (T\cap C)\ge 7$.
By Lemma \ref{nn1} we have $h^1(M,\Ii _{C,M}(4)) \ge 16$, i.e. $h^0(M,\Ii _{C,M}(4)) \ge 2$. Since $\alpha >3$ 
(Lemma \ref{wo6}), each $S\in |\Ii _{C,M}(4)|$ is irreducible.
 Let $\Bb$ denote the linear system on $W'$
induced by $|\Ii _{C,M}(4)|$ and fix a general $S\in |\Ii _{C,M}(4)|$. Write $S\cap W' = C+C' \in \Bb$. Since $C'$ is linked to $C$ by the complete intersection
$S\cap W'$, we have $\deg (C') =8$ and $p_a(C') =-10$ (\cite[Proposition 3.1]{ps}). Hence $C'$ has a multiple component. Since $W'$ contains no curve
of degree $x\in \{2,3\}$, the multiple component is a line. Since $W'$ has a unique line, $R$, $R$ is the multiple component. We saw that $R\cap C=\emptyset$
and $C\subset W_{\reg}$. Since $\dim (\Bb )>0$, $\Bb$ has at least one movable component, $A$. By Lemma \ref{non1} $A$ is a plane curve of degree $x\in \{4,5\}$.
We have $C\cup R \cup A \subset S$.
First assume $x=5$. Since $A\in |\Oo _S(1)|$, $C\cup R$ is contained in an element of $|\Oo _S(3)|$. Since the restriction map $H^0(M,\Oo _M(3)) \to H^0(S,\Oo _S(3))$
we get $\alpha \le 3$, contradicting Lemma \ref{wo6}.
Now assume $x=4$. Since $R$ is the only line of $W'$ we get $A\cup R\in |\Oo _S(1)|$. As above we get $\alpha \le 3$, a contradiction.

\providecommand{\bysame}{\leavevmode\hbox to3em{\hrulefill}\thinspace}

\end{document}